\documentclass[12pt,a4paper]{article}
\usepackage{amsmath}
\usepackage{amssymb}
\usepackage{amsthm}
\usepackage{amsfonts}
\usepackage{latexsym}
\usepackage{verbatim} 
\usepackage{tikz-cd}
\usepackage[pagewise]{lineno}%\linenumbers
\usepackage[a4paper,width=16cm,top=2cm,bottom=2cm,footskip=1cm]{geometry}

\usepackage{esvect}
\usepackage{mathrsfs}

\theoremstyle{plain}
\newtheorem{thm}{Theorem}[section]

\newtheorem{lem}{Lemma}[section]
\newtheorem{prop}{Proposition}[section]

\theoremstyle{definition}

\begin{document}

\title{Remarks on asymptotic isometric embeddings of conic transforms for torus actions}
\author{Andrea Galasso\footnote{\noindent{\bf Address:} Dipartimento di Matematica e Applicazioni, Universit\`a degli Studi di Milano-Bicocca, Via R.	Cozzi 55, 20125 Milano, Italy; {\bf ORCID iD:} 0000-0002-5792-1674; {\bf e-mail}: andrea.galasso@unimib.it; andrea.galasso.91@gmail.com}}

\date{}

\maketitle
	
\begin{abstract} 
	Consider a Hodge manifold and assume that a torus acts on it in a Hamiltonian and holomorphic manner and that this action linearizes on a given quantizing line bundle. Inside the dual of the line bundle one can define the circle bundle, which is a strictly pseudoconvex CR manifold. Then, there is an associated unitary representation on the Hardy space of the circle bundle. Under suitable assumptions on the moment map, we consider certain loci in unit circle bundle, naturally associated to a ray through an irreducible weight. Their quotients are called conic transforms. We introduce maps which are asymptotic embeddings of conic transforms making use of the corresponding equivariant Szeg\H{o} projector.  
\end{abstract}
	
	\tableofcontents
	
	\bigskip
	\textbf{Keywords:} symplectic manifold, group action
	
	\textbf{Mathematics Subject Classification:}  53D20, 32Q40
	
%\textbf{Data Availability Statements:} Data sharing not applicable to this article as no datasets were generated or analyzed during the current study.

\section{Introduction}

Let $(M,\,\omega)$ be a compact connected Hodge manifold of complex dimension $d$ with complex structure $J$ and quantum line bundle $(L,h)$ whose curvature form of the unique compatible and holomorphic connection $\nabla$ is $-2i\,\omega$; we shall denote by $g$ the Riemannian structure $\omega(\cdot\,,J\cdot)$. The volume form $\omega^{\wedge d}/d!$ is denoted by $\mathrm{dV}_M$. Let $X\subset L^{\vee}$ the unit circle bundle, with projection $\pi:X\rightarrow M$ and connection contact form $\alpha$ such that $\mathrm{d}\alpha =2\,\pi^*(\omega)$. Then $(X,\,\alpha)$ is a CR manifold and $\alpha$ is the contact form, see~\cite{z2}. There is a natural volume form on $X$ given by $\mathrm{dV}_X=(2\pi)^{-1}\alpha \wedge \pi^*(\mathrm{dV}_M)$. The Hardy space $H(X)$ represents the quantum space and under some suitable hypothesis classical symmetries on $M$ of a Lie group $G$ give rise to a quantum representation of $G$ on $H(X)$. Fix a coprime weight $\boldsymbol{\nu}$ labeling a unitary irreducible representation in $\mathfrak{g}^{\vee}$. The ladder $\mathcal{L}$ is the set of unitary irreducible representations labeled by integer elements of the ray $i\mathbb{R}_+\cdot\boldsymbol{\nu}$. In the setting of ladder representations, see \cite{gs}, one is led to consider spaces $H_{k\boldsymbol{\nu}}(X)$ for a fixed unitary irreducible representation $\boldsymbol{\nu}$ as $k \rightarrow +\infty$ and the corresponding projectors on it. The aim of this paper is to study geometric properties of these projectors.

More precisely, we shall consider the Abelian case: given a torus $T$ of dimension $n$ and a holomorphic and Hamiltonian action $\mu :T\times M\rightarrow M$, denote the moment map by $\Phi:M\rightarrow \mathfrak{t}^{\vee}$  (where $\mathfrak{t}$ and $\mathfrak{t}^{\vee}$ is the Lie algebra and co-algebra, respectively and we 
shall equivariantly identify $\mathfrak{t}\cong\mathfrak{t}^{\vee}\cong i\,\mathbb{R}^n$). The Hamiltonian action $\mu$ naturally induces an infinitesimal contact action of $\mathfrak{t}$ on $X$ (see~\cite{k}); explicitly, if $\boldsymbol{\xi}\in \mathfrak{t}$ and $$\boldsymbol{\xi}_M(m):=\mathrm{d}_{\mathrm{Id}}{\mu}_m(\boldsymbol{\xi})$$ is the corresponding Hamiltonian vector field on $M$ (here $\mu_m\,:\,T\rightarrow M$ is given by fixing $m\in M$ and its differential at the identity $\mathrm{Id}\in T$ is denoted by $\mathrm{d}_{\mathrm{Id}}{\mu}_m\,:\, \mathfrak{t}\rightarrow T_mM$) then its contact lift $\boldsymbol{\xi}_X$ is
\begin{equation} \label{eq:inf}
\boldsymbol{\xi}_X:= \mathrm{d}_{\mathrm{Id}}\tilde{\mu}_{\,\cdot}(\boldsymbol{\xi})= \boldsymbol{\xi}_M^\sharp-\langle \Phi\circ \pi, \boldsymbol{\xi}\rangle \,\partial_\theta\,,
\end{equation}
where $\boldsymbol{\xi}_M^\sharp$ denotes the horizontal lift on $X$ of a vector field $\boldsymbol{\xi}_M$ on $M$, and $\partial_\theta$ is the generator of the structure circle action on $X$ (similarly as before $\mathrm{d}_{\mathrm{Id}}\tilde{\mu}_x\,:\,\mathfrak{t}\rightarrow T_xX$ denotes the differential of the map $\tilde{\mu}_x\,:\,T\rightarrow X$ at the identity).

Furthermore, suppose that the infinitesimal action \eqref{eq:inf} can be integrated to an action $$\widetilde{\mu}:T\times X\rightarrow X$$
acting via contact and CR automorphism on $X$. We recall that, since $X$ is the boundary of the strictly pseudoconvex domain 
$$D = \{x \in L^{\vee}\,:\, \lvert x\rvert_h <1\} \subset L^{\vee}$$
it has a Szeg\H{o} kernel $\Pi(x,\, y)$ which projects $L^2(X)$ to the Hardy
space $H(X)$ of square summable functions which are boundary values of holomorphic functions in $D$. Under these assumptions, there is a naturally induced unitary representation of $T$ on the Hardy space
$H(X)\subset L^2(X)$ given by
\[\hat{\mu}_t(s)(x):= s(\tilde{\mu}_{t^{-1}(x)})\,,\qquad s\in H(X),\,x\in X\text{ and }t\in T\,. \]
Hence $H(X)$ can be equivariantly decomposed over the irreducible representations of $T$:
\[ H(X)=\bigoplus _{\boldsymbol{\nu}\in \widehat{T}} H(X)_{\boldsymbol{\nu}}, \]
where $\widehat{T}$ is the collection of all irreducible representations of $T$ which can be identified with $\mathbb{Z}^n$.
If $\Phi (m)\neq 0$ for every $m\in M$, then each isotypical component $H(X)_{\boldsymbol{\nu}}$
is finite dimensional, see \cite{pao-IJM}. For each $\boldsymbol{\nu}\in \widehat{T}$ we denote by
\[\chi_{\boldsymbol{\nu}}(t):=t^{\boldsymbol{\nu}}=e^{i\,\langle \boldsymbol{\nu},\,\boldsymbol{\theta}\rangle}\,,\qquad t\in T,\,\boldsymbol{\theta}\in \mathbb{R}^n\text{ and }\boldsymbol{\nu}\in\mathbb{Z}^n\]
the corresponding character; $\langle\cdot,\,\cdot \rangle$ denotes the standard scalar product on $i\,\mathbb{R}^n\cong \mathfrak{t}$. Here, we write $t$ the matrix exponential $e^{i\,\boldsymbol{\theta}}$, where $\boldsymbol{\theta}\in \mathbb{R}^n$. 

Fix $\boldsymbol{\nu}\in \widehat{T}$, let $\Pi_{\boldsymbol{\nu}}:L^2(X)\rightarrow H(X)_{\boldsymbol{\nu}}$ 
be the corresponding equivariant Szeg\H{o} projection; we also denote by $\Pi_{\boldsymbol{\nu}}\in \mathcal{C}^\infty (X\times X)$ its distributional kernel:
\[
	\Pi_{\boldsymbol{\nu}}(x,\,y):= \sum_{j=1}^{d_{\boldsymbol{\nu}}} s^{\boldsymbol{\nu}}_j(x)\cdot \overline{s^{\boldsymbol{\nu}}_j(y)}
\]
where $\{s^{\boldsymbol{\nu}}_j\}_{j=1}^{d_{\boldsymbol{\nu}}}$ is an orthonormal basis of $H(X)_{\boldsymbol{\nu}}$. It is given explicitly by the formula:
\[ \Pi_{\boldsymbol{\nu}}(x,\,y)=\int_T \overline{\chi_{\boldsymbol{\nu}}(t)}\,\Pi(\tilde{\mu}_{t^{-1}}(x),\,y)\,\mathrm{dV}_T(t) \label{eq:equivsz}
\]
where $\Pi$ is the  Szeg\H{o} kernel, $\mathrm{dV}_T$ is the Haar measure and $\chi_{\boldsymbol{\nu}}$ is the character of the representation $\boldsymbol{\nu}$. In~\cite{pao-IJM}, the dimensions of isotypes $H(X)_{k\boldsymbol{\nu}}$ and the behavior of $\Pi_{k\boldsymbol{\nu}}(x,\,x)$ are studied as $k\rightarrow +\infty$. In particular if $\boldsymbol{0}\notin \Phi(M)$ and $\Phi$ is transversal to the ray $i\mathbb{R}_+\cdot \boldsymbol{\nu}$ the asymptotic of $\Pi_{k\boldsymbol{\nu}}(x,\,x)$ is rapidly decreasing whenever $x\notin X_{\boldsymbol{\nu}}:=\pi^{-1}(M_{\boldsymbol{\nu}})$, where $M_{\boldsymbol{\nu}}:=\Phi^{-1}(i\,\mathbb{R}_+\cdot\boldsymbol{\nu})$, as $k$ goes to infinity. In \cite[Theorem 2]{pao-IJM} point-wise expansions and scaling asymptotics are considered at $x\in X_{\boldsymbol{\nu}}$. 

If $\boldsymbol{0}\notin \Phi(M)$ then the dimension of isotype $H(X)_{k\boldsymbol{\nu}}$ is finite. There is a canonical map $$\widetilde{\varphi_{k\boldsymbol{\nu}}}\,:\,X\rightarrow H_{k\boldsymbol{\nu}}(X)^{\vee}$$ given by $$x\mapsto \widetilde{\varphi_{k\boldsymbol{\nu}}}(x)=\Pi_{k\boldsymbol{\nu}}(x,\,\cdot)\,.$$ Let us consider an orthonormal basis of $H_{k\boldsymbol{\nu}}(X)$ given by  ${s}_1^{k\boldsymbol{\nu}},\,\dots,\,{s}^{k\boldsymbol{\nu}}_{d_{k\boldsymbol{\nu}}}$ with respect to the standard $L^2$ product. Now, as a consequence of the scaling asymptotics of the equivariant Szeg\H{o} kernel we can investigate geometric properties of the map
\[\widetilde{\varphi_{k\boldsymbol{\nu}}}\,:\, X_{{\boldsymbol{\nu}}} \rightarrow H_{k\boldsymbol{\nu}}(X)^{\vee}\,,\quad x \mapsto  \Pi_{k\boldsymbol{\nu}}(x,\,\cdot)=\left({s}_1^{k\boldsymbol{\nu}}(x),\,\dots,\,{s}^{k\boldsymbol{\nu}}_{d_{k\boldsymbol{\nu}}}(x)\right)\,, \]
as $k$ goes to infinity, where we write the image $\widetilde{\varphi_{k\boldsymbol{\nu}}}(x)$ in components with respect to the dual basis. 

An equivariant map $R \rightarrow S$ between two ${T}$-spaces descends to a map $R/{T}\rightarrow S/{T}$. Thus, $\widetilde{\varphi_{k\boldsymbol{\nu}}}$ is ${T}$-equivariant and it descends to a map ${\varphi_{k\boldsymbol{\nu}}}$ between the corresponding quotients:
\[\varphi_{k\boldsymbol{\nu}}\,:\, N_{\boldsymbol{\nu}} \rightarrow \mathbb{CP}^{d_{k\boldsymbol{\nu}}-1}_{\chi_{k\boldsymbol{\nu}}}\,, \]
here we denote by $N_{\boldsymbol{\nu}}$ the conic transform $X_{\boldsymbol{\nu}}/{T}$, see \cite{paopolorb}. Furthermore we shall recall in Section \ref{sec:kahlerN} from \cite{paopolorb} that it has a K\"ahler structure $\eta_{\boldsymbol{\nu}}$. The target space is given by taking the quotient of $H_{k\boldsymbol{\nu}}(X)^{\vee}\setminus \{\boldsymbol{0}\}$ with respect to the action induced by the representation
\[
	{s}\left(\tilde{\mu}_{t^{-1}}(x)\right)= {\chi}_{k\boldsymbol{\nu}}(t)\,{s}(x)
\]
where ${\chi}_{k\boldsymbol{\nu}}\,:\,{T}\rightarrow S^1$ is the character. Consider the symplectic form on 
$\mathbb{C}^{d_k}\setminus \{\boldsymbol{0}\}\cong H_{k\boldsymbol{\nu}}(X)^{\vee}\setminus \{\boldsymbol{0}\}$:
\[ \tilde{\omega}_{d_k}:=\frac{i}{2}\,\partial\overline{\partial}\log \lvert\zeta \rvert^2 \,.\]
The restriction of the action of ${T}$ on the unit sphere in $(\mathbb{C}^{d_k}\setminus \{\boldsymbol{0}\},\, \tilde{\omega}_{d_k})$ has the same orbit of the standard $S^1$ action on $S^{d_k-1}$, $\tilde{\omega}_{d_k}$ descends to the Fubini Study form $\omega_{FS,\,k}$ on $\mathbb{CP}^{d_{k\boldsymbol{\nu}}-1}_{\chi_{k\boldsymbol{\nu}}}$.

Let us denote by ${T}_x$ the stabilizer of a point $x\in X_{\boldsymbol{\nu}}$. The cardinality $\lvert {T}_x\rvert$ need not be constant on $X_{\boldsymbol{\nu}}$, but it does attain a generic minimal value $\lvert {T}_{\mathrm{min}}^{\boldsymbol{\nu}}\rvert$ on some dense open subset $X_{\boldsymbol{\nu}}'$, where ${T}_{\mathrm{min}}^{\boldsymbol{\nu}} \subseteq {T}$ is the stabilizer of each point in $X_{\boldsymbol{\nu}}'$ (Corollary $B.47$ of~\cite{ggk}). Then ${T}_{\mathrm{min}}^{\boldsymbol{\nu}}$ stabilizes every $x\in X_{\boldsymbol{\nu}}$. Suppose that $x\in X_{\boldsymbol{\nu}}\setminus X_{\boldsymbol{\nu}}'$ and consider the projection $p\,:\,X_{\boldsymbol{\nu}}\rightarrow N_{\boldsymbol{\nu}}$, then $p(x)$ is a singular point in $(N_{\boldsymbol{\nu}})_{\mathrm{sing}}\subseteq N_{\boldsymbol{\nu}}$.

If $k$ is sufficiently large, we shall prove that $\varphi_{k\boldsymbol{\nu}}$ is a asymptotically rescaled K\"ahler embedding, in a sense specified below. By making use of a variant of \cite[Theorem 4]{pao-IJM} we can prove the following theorem, which is an analogue of embedding theorems in K\"ahler geometry, see \cite{z2}. For terminology concerning orbifolds, see Section \ref{sec:orb}.

\begin{thm} \label{cor:aimi}
	Assume that $\boldsymbol{\nu}\neq \boldsymbol{0}$ is coprime, $\boldsymbol{0}\notin \Phi(M)$, ${T}_{\mathrm{min}}^{\boldsymbol{\nu}}$ is trivial and $\Phi$ is transversal to $i\,\mathbb{R}_+\cdot \boldsymbol{\nu}$. Then there exists a subsequence $k_n$ such that the maps $\varphi_{k_n\boldsymbol{\nu}}$ form an asymptotic sequence of K\"ahler embeddings of the K\"ahler orbifold $(N_{\boldsymbol{\nu}},\,\eta_{\boldsymbol{\nu}})$ into $(\mathbb{CP}^{d_{k\boldsymbol{\nu}}-1}_{\chi_{k\boldsymbol{\nu}}},\,\omega_{FS,k})$ in the following sense: 
	\begin{itemize}
		\item[(1)] $\varphi_{k_n\boldsymbol{\nu}}$ is an \textit{orbifold embedding} if $n$ is sufficiently large. 
		\item[(2)] Uniformly on compact subsets of the smooth locus $N_{\boldsymbol{\nu}}\setminus(N_{\boldsymbol{\nu}})_{\mathrm{sing}}$, there exists $\epsilon>0$ such that $\varphi_{k_n\boldsymbol{\nu}}$ satisfies:
		$${\varphi_{k\boldsymbol{\nu}}}^{*}(\omega_{FS,k})=\eta_{\boldsymbol{\nu}}+O(k^{-\epsilon})\,.$$ 
		%\item[(3)] it defines asymptotic minimal immersion into spheres:	$$\Delta_{g_2} {\varphi_{k\boldsymbol{\nu}}}(x)=-(2d-2r_G+2)\,{\varphi_{k\boldsymbol{\nu}}}(x)+O(k^{-\epsilon})\,.$$
	\end{itemize}
\end{thm}

%In \cite{p-new} scaling asymptotics were investigated by making use of the Kirillov character formula. In fact, Theorem \ref{thm:tang} is a simple variant of the main Theorem in \cite{p-new}.

Theorem~\ref{thm:tang} is based on micro-local techniques that work in the almost complex symplectic setting, see~\cite{sz}. Following similar idea as in~\cite{sz}, we study off-locus asymptotic of the projector $\Pi_{k\boldsymbol{\nu}}$ for arbitrary displacement as $k$ goes to infinity, we refer to \cite{pao-IJM} and Section \ref{sctn:szego fio}.

%Eventually we recall that, motivated by K\"ahler geometry, various authors studied canonical isometric embeddings of Riemannian manifolds into $\mathbb{S}^{k-1}$ or $\mathbb{R}^k$, for $k \gg 1$. In 1994, B\'erard, Besson, and Gallot \cite{bbg} constructed an asymptotically isometric embedding of compact Riemannian manifolds $M$ into the space of real-valued and square summable series, see also \cite{z1}, \cite{wz}, \cite{po}. Those approaches are somehow close in spirit with the one developed in this paper, in fact the microlocal expression for Szeg\H{o} kernel of a CR manifold $X$ is closely related to the heat kernel of the Kohn Laplacian on it, see the proof in \cite{hsiao}.

\section{Motivations and preliminaries}
\label{sec:premot}

Given a compact Hodge manifold $(M,\omega)$ the sections of the quantizing line bundle $L$ can be used to define an embedding ${\phi}_k\,:\, M\rightarrow \mathbb{CP}^{d_k-1}$ into the complex projective space $(\mathbb{CP}^{d_k-1},\,\omega_{FS})$ for $k$ large, where $\omega_{FS,k}$ is the Fubini-Study form and $d_k$ is the dimension of the space of holomorphic sections of the $k$-th tensor power of the line bundle $L$. These spaces of sections can be naturally identified with spaces of functions on $X$ lying in the $k$-th Fourier component of the Hardy space.  The map $\phi_k$ can be studied by using the Szeg\H{o} kernel, in fact its $k$-th Fourier component defines, for $k$ sufficiently large, a map $$\widetilde{\phi}_k\,:\,X\rightarrow H_k(X)^{*}\setminus\{\boldsymbol{0}\}\,,\quad x\mapsto \Pi_k(x,\,\cdot)\,,$$ 
which in turn define a map
\[ \phi_k\,:\,M\rightarrow \mathbb{P}(H_k(X)^{*})\,. \]
It is natural to ask how the pull-back of $\omega_{FS}$ on $M$ is related to $\omega$. S. Zelditch and D. Catlin strengthen a previous result of Tian by proving that $(1/k)\,{\phi_k}^*(\omega_{FS})$ converges to $\omega$ in the smooth topology, see \cite{cat} and \cite{z2}. Furthermore, B. Shiffman and S. Zelditch generalize this picture in the symplectic setting in~\cite{sz}. In this setting, small displacements from a fixed $x\in X$ are conveniently expressed
in Heisenberg local coordinates on $X$ centered at $x$. A choice of Heisenberg local coordinates at $x$ gives a meaning to the expression $x+ \mathbf{v}$, where $\mathbf{v} \in T_{\pi(x)}M$ has sufficiently small norm. By using re-scaling asymptotic expansion
\[k^{-d}\,\Pi_k\left(x+\frac{\mathbf{v}}{\sqrt{k}},\,x+\frac{\mathbf{w}}{\sqrt{k}} \right)\sim\frac{1}{\pi^d}e^{\mathbf{v}\cdot \mathbf{w}-\frac{1}{2}(\lVert \mathbf{v}\rVert^2+\lVert \mathbf{w}\rVert^2)}\left[1+\sum_{j\geq 1} k^{-j/2}a_j(x,\, \mathbf{v},\,\mathbf{w}) \right] \]
where $a_j(x,\, \mathbf{v},\,\mathbf{w})$ are polynomial functions in $\mathbf{v}$ and $\mathbf{w}$ and $\mathbf{v},\,\mathbf{w}\in T_{m_x}M$, they prove the convergence of the rescaled pull-backs of the Fubini Study forms. Here, $\sim$ means ``as the same asymptotic as''.
Furthermore in \cite{gh}, we obtain an analogue Kodaira embedding theorem and Tian's almost-isometry theorem for compact quantizable pseudo-K\"ahler manifold making use of the asymptotic expansion of the Bergman kernels for $L^{\otimes k}$-valued $(0,q)$-forms.

In this paper we shall focus on  Szeg\H{o} kernel asymptotics along rays in weight space and we aim to obtain analogue of results we have discussed in this section, we shall now briefly review the literature. In \cite{pao-IJM} the author focuses on asymptotic expansions locally reflecting the equivariant decomposition of $H(X)$ over the irreducible representations of a compact torus. In \cite{gp} and \cite{gp-asian} the focus was on respectively the case $G=U(2)$ and $G=SU(2)$ making use of the Weyl integration formula; in \cite{p-new} local scaling asymptotics were generalized to the case of a connected compact Lie group $G$. Furthermore, we refer to \cite{circle} where asymptotics of the associated Szeg\H{o} and Toeplitz operators were studied in detail for the case of Hamiltonian circle action lifting to the quantizing line bundle, see \cite{g} for some related results concerning the torus action case. It is important to notice that, in general, the spaces $H(X)_{k\boldsymbol{\nu}}$ are not contained in any space of global sections $H^0(M,A^{\otimes k})$. 

The analog of ``quantization commutes with reduction'' for torus action in this framework was first studied in \cite{paopolorb}. In \cite{g2} we generalize the main theorem in \cite{paopolorb} to compact “quantizable” pseudo-Kähler manifolds.

\subsection{Definitions concerning orbifolds}
\label{sec:orb}

We recall basic definitions we need about orbifolds, for a more precise discussion see \cite{s}, \cite{alr}, \cite{mm} and references therein. We define the category $\mathcal{M}_s$ in the following way. The objects are pairs $(U, G)$, where $U$ is a connected smooth manifold and $G$ is a finite group acting smoothly on $U$. A morphism $\Phi\,:\,(U,G)\rightarrow (U',G')$ is family of open embeddings $\varphi:\,U\rightarrow U'$ such that:
\begin{itemize}
	\item[i)] For each $\varphi\in \Phi$ there exists an injective group homomorphism $\lambda_{\varphi}: G\rightarrow G'$ satisfying
\[\varphi(g\cdot x)=\lambda_{\varphi}(g)\cdot \varphi(x),\qquad x\in U,\,g\in G \,. \]
	\item[ii)] For $g\in G'$, $\varphi\in \Phi$, if $(g\,\varphi)(U)\cap \varphi(U)\neq \emptyset$, then $g\in \lambda_{\varphi}(G)$.
	\item[iii)] For $\varphi\in \Phi$, we have $\Phi=\{g\,\varphi,\,g\in G'\}$.
\end{itemize} 

Let $X$ be a paracompact Hausdorff space and let $\mathcal{U}$ be a covering of $X$ consisting of connected open subsets. We assume $\mathcal{U}$ satisfies the condition: for any $x\in U\cap V$, $U,\,V\in\mathcal{U}$, there is $W\in \mathcal{U}$ such that $x\in W \subset U\cap V$. Then we say that the topological space $X$ has an orbifold structure if 
\begin{itemize}
	\item[i)] for each $U\in \mathcal{U}$ there exists $(\widetilde{U},G_U)\in \mathcal{M}_s$ and a ramified covering $(\widetilde{U},G_U)\rightarrow U$ such that $U$ is homeomorphic to $\widetilde{U}/G_U$, where $\widetilde{U}/G_U$ is endowed with the quotient topology;  
	\item[ii)] for any $U,\,V \in \mathcal{U}$, $U\subset V$, there exists a morphism $\varphi_{V,U}: (\widetilde{U},G_U) \rightarrow (\widetilde{V},G_V)$ which covers the inclusion  $U\subset V$ and satisfies $\varphi_{WU}= \varphi_{WV}\circ \varphi_{VU}$ for any $U,\,V,\,W \in \mathcal{U}$ such that $U\subset V \subset W$.
\end{itemize}

Let us denote by $\pi_U\,:\, (\widetilde{U},G_U)\rightarrow U$ the ramified covering. We recall that for each $p\in X$ one can always find local coordinates $x=(x_1,\,\dots,\,x_n)$ on $(\widetilde{U}, G_U)$, $\widetilde{U}\subseteq \mathbb{R}^n$, and $\tilde{p}$ with $\pi_{U}(\tilde{p})=p$ such that $x(\tilde{p})\equiv 0\in \mathbb{R}^n$ is a fixed point of $G_U$ and $G_U$ acts linearly on $\mathbb{R}^n$. We call these coordinates \textit{standard coordinates} at $p$. If $\vert G_U\rvert > 1$ then we say that $p$ is a singular point, otherwise we say that $p$ is regular. 

%There exists an open and dense subset in $X$ on which all the points have conjugated stabilizers, thus if the action of $G$ on $M$ in the definition of $\mathcal{M}_s$ is assumed to be effective, then set of regular points is open and dense in $X$. 

%\begin{exmp} Consider the action of $(\mathbb{Z}_2,+)$ on $\mathbb{R}^3$ given by \[[0]_2\cdot (x_1,x_2,x_3)= (x_1,x_2,x_3),\quad [1]_2\cdot (x_1,x_2,x_3)= (x_1,x_2,-x_3) \, \] and the action of $(\mathbb{Z}_3,+)$ is \[[1]_3\cdot (x_1,x_2, x_3)= \left(x_1,-\frac{1}{2}x_2 - \frac{\sqrt{3}}{2}x_3,-\frac{1}{2}x_3 + \frac{\sqrt{3}}{2}x_2 \right) \, \] and $[2]_3\cdot x = [1]_3\cdot [1]_3 \cdot x$. Consider the three dimensional sphere $S^3$, the the north pole $N$, the south pole $S$ and the corresponding stereographic projections \[\mathfrak{s}_N\,: S^3\setminus \{N\} \rightarrow \mathbb{R}^3,\quad \mathfrak{s}_S\,: S^3\setminus \{S\} \rightarrow \mathbb{R}^3 \]  \end{exmp}

A \textit{vector orbibundle} $(E,p)$ on an orbifold $X$ is an orbifold $E$ together with a continuous projection $p:E\rightarrow X$ such that for $U\in \mathcal{U}$ there exists a $G_U$-equivariant lift $\widetilde{p}\,: \widetilde{E}_{U}\rightarrow \widetilde{U}$ defining a trivial vector bundle such that the diagram
\begin{align*}
	(\widetilde{E}_U,&G^E_U) \xrightarrow{\widetilde{p}} (\widetilde{U}, G_U)   \\
	\downarrow& \qquad\qquad \downarrow \\
	E_U& \quad\,\,\,\,\xrightarrow{p}  \quad U
\end{align*}
commutes. As before, when we say that $\widetilde{p}\,: \widetilde{E}_{U}\rightarrow \widetilde{U}$ is $G_U$-equivariant we mean that there exists an injective group homomorphism $\lambda_{p}: G^E_U \rightarrow G_U$ such that
\[\widetilde{p}(g\cdot e)=\lambda_{p}(g)\cdot \widetilde{p}(e)\qquad e\in \widetilde{E}_U,\,g\in G_U^E \,;\]
if $\lambda_{p}$ is a group isomorphism, we say that the vector orbibundle is \textit{proper}.

A smooth section of a vector orbibundle $(E,p)$ over $X$ is a smooth map $s: X \rightarrow E$ such that $p\circ s = 1_X$, we denote the space of smooth sections as $\mathcal{C}^\infty(X, E)$. We recall that a \textit{smooth map} $f:X\rightarrow Y$ between orbifolds is a continuous map between the underlying topological spaces such that for each $U\in \mathcal{U}$ there exists a local lift $(\widetilde{f}_U, \overline{f}_U)$ such that $\widetilde{f}_U : \widetilde{U}\rightarrow \widetilde{V}$ is smooth, $\overline{f}_U: G_U\rightarrow G_V$ is a group homomorphism, the diagram
\begin{align*}
	(\widetilde{U},& G_U) \xrightarrow{\widetilde{f}} (\widetilde{V},G_V) \\
	\downarrow& \qquad\qquad \downarrow \\
	U& \quad\,\,\,\,\xrightarrow{f} \quad V
\end{align*}
commutes and $\widetilde{f}_U$ is $\overline{f}_U$-equivariant:
\[\widetilde{f}_U(g\cdot x)= \overline{f}_U(g)\widetilde{f}_U(x)\,,\qquad g\in G_U,\,x\in\widetilde{U}\,. \]

The tangent orbibundle $TX$ is defined by gluing together the bundles defined over the charts 
\[(T\widetilde{U},\,G_U) \rightarrow (\widetilde{U},\,G_U)\,,  \]
where the action of $G_U$ on $T_{\widetilde{x}}\widetilde{U}$ is induced by differentiating the action on $\widetilde{U}$, $\mathrm{d}g$. It has a natural structure of proper vector orbibundle. The smooth sections of it are called vector fields.

We can then define the \textit{differential} of a smooth map $f : X \rightarrow Y$ to be the smooth map $\mathrm{d}f: TX \rightarrow TY$ such that locally for a given $x\in U$, $f(x)\in V$, and \textit{standard coordinates} at $x$ we have a commutative diagram
\begin{align*}
	(T_{\widetilde{x}}\widetilde{U},& G_U) \xrightarrow{\mathrm{d}_{\widetilde{x}}\widetilde{f}} (T_{\widetilde{f}(\widetilde{x})}\widetilde{V},G_V) \\
	\downarrow& \qquad\qquad \,\,\,\downarrow \\
	T_xU& \quad\,\,\,\,\xrightarrow{\mathrm{d}_xf} \,\, T_{f(x)}V
\end{align*}
where $\mathrm{d}_{\widetilde{x}}\widetilde{f}$ is $\overline{f}_U$-equivariant. We say that $f$ is an \textit{immersion} at $x \in X$ if $\mathrm{d}f$ is injective, so this means that there is a local lift $\widetilde{f}_U : \widetilde{U} \rightarrow \widetilde{V}$ such that both $\mathrm{d}_{\widetilde{x}}\widetilde{f}$ and $\overline{f}_U$ are injective. 

\subsection{Conic transforms for torus actions}
\label{sec:kahlerN}

We work with the same assumptions as in \cite[Basic Assumption 1.1]{paopolorb}. Let us recall that $M_{\boldsymbol{\nu}}=\Phi^{-1}(i\mathbb{R}_+\cdot\boldsymbol{\nu})$ and for each $m\in M_{\boldsymbol{\nu}}$ we have
\[\Phi(m)=\varsigma(m)\,i\,\boldsymbol{\nu}\,,\]
where $\varsigma\,:\,M\rightarrow \mathbb{R}$ is a suitable smooth positive function.
Here, we always assume that the action of $T$ on $X_{\boldsymbol{\nu}}$ is locally free. By \cite{paopolorb}, transversality is tantamount to $T$ acting locally freely on $X_{\boldsymbol{\nu}}$ so that $N_{\boldsymbol{\nu}}$ is an orbifold.

In \cite[Section 5.3]{paopolorb} a K\"ahler orbifold structure for $N_{\boldsymbol{\nu}}$ is defined in the following way. Let us denote by $i\,\boldsymbol{\nu}^0$ the annihilator of $i\,\boldsymbol{\nu}$ in $\mathfrak{t}$. Consider the quotient $Y_{\boldsymbol{\nu}}:=X_{\boldsymbol{\nu}}/\exp_T(i\,\boldsymbol{\nu}^0)$, then $Y_{\boldsymbol{\nu}}$ is as principal $S^1$ orbibundle over $N_{\boldsymbol{\nu}}$, let $\pi^{\boldsymbol{\nu}}$ be the projection. Adopting the same notation as in \cite[Section 5.3]{paopolorb}, $\Phi\circ \pi$ descends to a smooth function $\overline{\Phi}\,:\,Y_{\boldsymbol{\nu}}\rightarrow \mathfrak{t}^{\vee}$; hence $\Phi^{\boldsymbol{\nu}}=\langle \Phi,\,i\,\boldsymbol{\nu}\rangle$ descends to a smooth function $\overline{\Phi}^{\boldsymbol{\nu}}\,:\,Y_{\boldsymbol{\nu}}\rightarrow \mathbb{R}$. 

Let $\alpha^{X_{\boldsymbol{\nu}}}:=j_{\boldsymbol{\nu}}^*(\alpha)$, where $j_{\boldsymbol{\nu}}\,:\,X_{\boldsymbol{\nu}}\rightarrow X$ is the inclusion. Then, $\alpha^{X_{\boldsymbol{\nu}}}$ is ${T}$-invariant, and by definition of $X_{\boldsymbol{\nu}}$ for any $\boldsymbol{\xi}\in \boldsymbol{\nu}^{0}$ we have $\iota((i\,\boldsymbol{\xi})_{X_{\boldsymbol{\nu}}})\alpha^{X_{\boldsymbol{\nu}}}=0$, see \cite{paopolorb}. Hence $\alpha^{X_{\boldsymbol{\nu}}}$ is the pullback of an orbifold 1-form $\alpha^{Y_{\boldsymbol{\nu}}}$ on $Y_{\boldsymbol{\nu}}$. Let us define
\[
	\beta_{\boldsymbol{\nu}}:=  \frac{\lVert\boldsymbol{\nu} \rVert^2}{\overline{\Phi}^{\boldsymbol{\nu}}}\, \alpha^{Y_{\boldsymbol{\nu}}}\,.
\]
 
By \cite[Equation (46)]{paopolorb}, we see that
\[\mathrm{d}\beta_{\boldsymbol{\nu}}=\lVert \boldsymbol{\nu} \rVert^2\, \left[\frac{1}{\overline{\Phi}^{\boldsymbol{\nu}}}\,\mathrm{d}\alpha^{Y_{\boldsymbol{\nu}}}-\frac{1}{(\overline{\Phi}^{\boldsymbol{\nu}})^2}\,\mathrm{d}\overline{\Phi}^{\boldsymbol{\nu}}\wedge \alpha^{Y_{\boldsymbol{\nu}}}\right]\,. \]

By \cite[Corollary 16]{paopolorb} and \cite[Lemma 11]{paopolorb}:
\begin{prop}
	There exists a $2$-form $\eta_{\boldsymbol{\nu}}$ on $N_{\boldsymbol{\nu}}$ such that $\mathrm{d}\beta_{\boldsymbol{\nu}}=2\pi^*_{\boldsymbol{\nu}}(\eta_{\boldsymbol{\nu}})$ and a complex structure $J^{N_{\boldsymbol{\nu}}}$ so that $(N_{\boldsymbol{\nu}},\,J^{\boldsymbol{\nu}},\,\eta_{\boldsymbol{\nu}})$ is a K\"ahler orbifold.
\end{prop}

First, we introduce a decomposition of $T_{m}M$ when $m\in M_{\boldsymbol{\nu}}$, as in \cite{pao-IJM}. We adopt the following notation. Let $\mathfrak{l}\leq \mathfrak{t}$ be a Lie subalgebra of $\mathfrak{t}$, we denote by $\mathfrak{l}_M(m)$ the space of infinitesimal vector fields $\boldsymbol{\xi}_M(m)$ on $M$ at $m$, $\boldsymbol{\xi}\in \mathfrak{l}$. Explicitly,
\[\mathfrak{l}_M(m)=\{\boldsymbol{\xi}_M(m)\in T_mM\,:\, \boldsymbol{\xi} \in \mathfrak{l}\}\,. \]
Similarly we denote by $\mathfrak{l}_X(x)$ the space of infinitesimal vector fields $\boldsymbol{\xi}_X(x)$ on $X$ at $x$, $\boldsymbol{\xi}\in \mathfrak{l}$. Under the transversality assumption, for each $m\in M_{\boldsymbol{\nu}}$ we have
\begin{equation} T_mM= T^{\mathrm{t}}_mM\oplus T^{\mathrm{v}}_mM \oplus T^{\mathrm{hor}}_mM\,, \label{eq:decomposition} 
\end{equation}
where
\[T^{\mathrm{t}}_m M:=J_m\,(i\boldsymbol{\nu}^{0})_M(m)\,, \qquad T^{\mathrm{v}}_mM:=(i\boldsymbol{\nu}^{0})_M(m)\,\]
and
\[T^{\mathrm{hor}}_mM := \left(T^{\mathrm{v}}_mM \oplus T^{\mathrm{t}}_mM \right)^{\perp_{g}}\,. \]

By the proof of \cite[Lemma 11]{paopolorb}, see \cite[Equation (48)]{paopolorb}, we have the following lemma.

\begin{lem} \label{lem:2.1} Let $Q_{\boldsymbol{\nu}}\,:\,X_{\boldsymbol{\nu}}\rightarrow Y_{\boldsymbol{\nu}}$ be the projection. For each $x\in X_{{\boldsymbol{\nu}}}$, with $\pi(x)=m$, and for each $\mathbf{v},\mathbf{w}\in T_{m}^{\mathrm{hor}}M$, we have
	\[Q^*_{\boldsymbol{\nu}}(\mathrm{d}\beta_{\boldsymbol{\nu}})_x(\mathbf{v}^{\sharp},\,\mathbf{w}^{\sharp})=\frac{2}{\varsigma(m)}\,\pi^*\omega_m(\mathbf{v}^{\sharp},\,\mathbf{w}^{\sharp})\,.\]	
\end{lem}

\subsection{Equivariant Szeg\H{o} kernels}
\label{sctn:szego fio}

Let $\Pi\,:\,L^2(X)\rightarrow H(X)$ be the Szeg\H{o} projector,  $\Pi(\cdot,\cdot)$ its kernel. By~\cite{bs}, $\Pi$ is a Fourier integral operator with complex phase:
\[
	\Pi(x,y)=\int_0^{+\infty}e^{iu\psi(x,y)}\,s(x,y,u)\,\mathrm{d}u,
\]
where the imaginary part of the phase satisfies $\Im (\psi)\ge 0$ and 
$$
s(x,y,u)\sim \sum_{j\ge 0}u^{d-j}\,s_j(x,y).
$$
We shall also make use of the description of the phase $\psi$ in Heisenberg local coordinates (see \S 3 of~\cite{sz}).

The Heisenberg coordinates for the unit circle bundle $X$ at a given point $x$, with $\pi(x)=m$, are defined in Section~\S$1.2$ in~\cite{sz}. The Heisenberg coordinates play a crucial role in the description of the scaling asymptotics of the Szeg\H{o} kernel in \cite{sz}. Here, we will only recall that to define them one needs to choose \textit{preferred local coordinates} $z=(z_1,\dots,z_d)$ centered at a point $m\in M$ and choose a \textit{preferred local frame} for the line bundle $L$. Recall that the coordinates $(z_1,\dots,z_d)$ are preferred at $m\in M$ if and only if
\[\sum_{j=1}^d \mathrm{d}z_j\otimes \mathrm{d}\overline{z}_j = (g-i\,\omega)_{\vert m}\,, \]
where $g$ is the Riemannian metric $\omega(\cdot\,,\, J\cdot)$. Furthermore, a preferred frame for $L \rightarrow M$ at a point $m \in M$ is a local frame $e_L$ in a neighborhood of $m$ such that
\[(\lVert e_L \rVert_{h_L})_{\vert m}=1\,,\quad \nabla (e_L)_{\vert m}=0\,,\]
and
\[\nabla^2(e_L)_{\vert m}=-(g+i\,\omega)\otimes (e_L)_{\vert m}\,. \]

The preferred frame and preferred coordinates together give us the Heisenberg local coordinates on the circle bundle $X$: a Heisenberg coordinate chart at a point $x$ ($m=\pi(x)$) is a coordinate chart $\rho\,:\,U\approx V$ with $0\in U\subset \mathbb{C}^d\times \mathbb{R}$ (denote by $p_{\mathbb{C}^d}\,:\,\mathbb{C}^d\times \mathbb{R}\rightarrow \mathbb{C}^d$ the projection) and $\rho(0)=x\in V\subset X$ of the form
\[\rho(\theta, z_1,\dots,z_d)=e^{i\theta}\,a(z)^{-\frac{1}{2}}\,e^*_L(z)\,, \]
where $a$ is a real-valued smooth function on $p_{\mathbb{C}^d}(U)$, $e_L$ is a preferred local frame at $m$ and $(z_1,\dots,z_d)$ are preferred coordinates centered at $m$.

For ease of notation, on a Heisenberg local chart, we write $x+(\theta,\,\mathbf{v})$; here $\theta\in (-\pi,\,\pi)$ and $\mathbf{v}\in B_{2d}(\boldsymbol{0},\,\delta)$, the open ball of center the origin and radius $\delta>0$ in $\mathbb{C}^d\cong \mathbb{R}^{2d}$. If $\mathbf{v}\in T_mM$ and $\lVert \mathbf{v}\rVert< \epsilon$, we write
$x+(0,\,\mathbf{v})=x+\mathbf{v}$.
Furthermore, accordingly with the decomposition \eqref{eq:decomposition}, we write $\mathbf{v}=\mathbf{v}_{\mathrm{t}}+ \mathbf{v}_{\mathrm{v}}+ \mathbf{v}_{\mathrm{hor}}$.

Now, we need to state a variant, Theorem \ref{thm:tang} below, of the main result concerning off-diagonal asymptotics of the equivariant Szeg\H{o} kernel obtained in~\cite{pao-IJM} and \cite{cm}. In fact, \cite[Theorem 4]{pao-IJM} concerns off-locus asymptotics expansions along directions in the normal bundle $T^{\mathrm{t}}M$ but one can generalize them for arbitrary displacements as in~\cite{cm}. More precisely, in \cite{cm}, the author considers that a compact and connected Lie group $G$ and a compact torus $T$ act on $M$ in a holomorphic and Hamiltonian manner, that the
actions commute, and linearize to $L$. Given a nonzero integral weight $\nu$ for $T$, the author considers the isotypical components associated to the multiples $k\,\nu_T$, $k\rightarrow +\infty$, and focuses on how their structure as $G$-modules is reflected by certain local scaling asymptotics on $X$. Thus, \cite[Theorem 1.6]{cm} reduces to Theorem \ref{thm:tang} below for trivial $G$.

Choose $\boldsymbol{\xi}_1\in \mathrm{span}(\boldsymbol{\nu})$ so that $\lVert \boldsymbol{\xi}_1\rVert=1$ and $\langle \boldsymbol{\xi}_1,\,\boldsymbol{\nu}\rangle=\lVert \boldsymbol{\nu} \rVert$. Hence, if $v_j=(\theta_j,\,\mathbf{v}_j)\in T_xX$, set \[ \mathbf{A}(v_1,\,v_2):=\frac{\theta_2-\theta_1}{\lVert \Phi(m) \rVert}\,\boldsymbol{\xi}_{1,\,M}(m)\,. \]Define $E\,:\,T_xX_{\boldsymbol{\nu}}\times T_xX_{\boldsymbol{\nu}} \rightarrow \mathbb{C}$ by setting
\[E(v_1,\,v_2):=i\,\omega_m(\mathbf{A}(v_1,\,v_2),\,\mathbf{v}_{1\,h})+\psi_2(\mathbf{v}_{1\mathrm{h}}-\mathbf{A}(v_1,\,v_2),\,\mathbf{v}_{2\,h}) \]
where
\[
\psi_2\left(\mathbf{v}_{1},\,\mathbf{v}_{2} \right) =-i\,\omega_m(\mathbf{v}_1,\,\mathbf{v}_2)-\frac{1}{2}\lVert \mathbf{v}_1-\mathbf{v}_2\rVert^2\,.
\]
Given $v=(\theta,\,\mathbf{v})\in T_xX$, for every $t\in T_x$, we write
\[v^{(t)}=\mathrm{d}_x\tilde{\mu}_{t}(v)=(\theta,\,\mathrm{d}_{m_x}{\mu}_{t}(\mathbf{v}))\,.\]

Now, we can state the theorem concerning off diagonal scaling asymptotics for directions tangent to $X_{\boldsymbol{\nu}}$.

\begin{thm} \label{thm:tang}
	For each $x\in X_{\boldsymbol{\nu}}$, with $\pi(x)=m$, and $v_j=(\theta_j,\,\mathbf{v}_j)\in T_xX_{\boldsymbol{\nu}}$, $j=1,\,2$, such that $\mathbf{v}_j\in T^{\mathrm{hor}}_mM$ and $\lVert \mathbf{v}_i \rVert\leq C\,k^{\epsilon}$, for sufficiently small $\epsilon$, we have
	\begin{align} \label{eq:piknu}
		&\Pi_{k\boldsymbol{\nu}}\left(x+\frac{{v}_{1}}{\sqrt{k}},\,x+\frac{{v}_2}{\sqrt{k}} \right)  \\
		&\qquad\sim  \frac{e^{i\,\sqrt{k}\,(\theta_1-\theta_2)/\varsigma(m)}}{(\sqrt{2}\,\pi)^{n-1}\,\mathcal{D}(m)}\,\left(\lVert \boldsymbol{\nu} \rVert\cdot\frac{k}{\pi}  \right)^{{d}+(1-{n})/2}\cdot\left( \frac{1}{\lVert \Phi(m) \rVert}  \right)^{{d}+1+(1-{n})/2} \notag\\
		&\qquad\quad\cdot  \left(\sum_{t\in T_x}\chi_{\boldsymbol{\nu}}(t)^k\,e^{\frac{1}{\varsigma(m)}\,E\left({v}_1^{(t)},\,{v}_2\right)}\right) \cdot \left(1+\sum_{j\geq 1}R_j(m,\,{v}_1,\,{v}_2)\,k^{-j/2} \right), \notag
	\end{align}
	where $\mathcal{D}$ is a smooth positive function on $M_{\boldsymbol{\nu}}$, $R_j$ is polynomial in the ${v}_i$'s of degree $\leq 3\, j$ and parity $j$.
\end{thm}

The re-scaled asymptotic expansions \eqref{eq:piknu} of Theorem \ref{thm:tang} will be the main tool for proving Theorem~\ref{cor:aimi}.

Let us compare the exponent $E\left({v}_1,\,{v}_2\right)$ with the standard one appearing in \cite{sz}. For ease of notation, we set
\begin{align*}&F_k^{\mu}\left(\left(\frac{\theta_1}{\sqrt{k}},\,\frac{\mathbf{v}_1}{\sqrt{k}}\right),\,\left(\frac{\theta_2}{\sqrt{k}},\,\frac{\mathbf{v}_2}{\sqrt{k}}\right)\right):= i\,\sqrt{k}\,({\theta_1-\theta_2})+E\left({v}_1,\,{v}_2\right)  \\
	&\qquad= i\,\sqrt{k}\,\left[\theta_1\left(1-\omega\left(\frac{\boldsymbol{\xi}_{1,\,M}}{\lVert\Phi(m)\rVert},\,\mathbf{v}_{1,h}\right)\right)-\theta_2\left(1-\omega\left(\frac{\boldsymbol{\xi}_{1,\,M}}{\lVert\Phi(m)\rVert},\,\mathbf{v}_{2,\,h}\right)\right) \right] \\
	&\qquad\quad +\psi_2\left(\mathbf{v}_{1,\,h}+\theta_1\cdot\frac{\boldsymbol{\xi}_{1,\,M}}{\lVert\Phi(m)\rVert},\,\mathbf{v}_{2,\,h}+\theta_2\cdot \frac{\boldsymbol{\xi}_{1,\,M}}{\lVert\Phi(m)\rVert}\right)
\end{align*}
and we note that, when $T=T^1$ acts trivially on $M$, with constant moment map equal to $1$, we recover the exponent of \cite{sz}:
\[F_k^{SZ}\left(\left(\frac{\theta_1}{\sqrt{k}},\,\frac{\mathbf{v}_1}{\sqrt{k}}\right),\,\left(\frac{\theta_2}{\sqrt{k}},\,\frac{\mathbf{v}_2}{\sqrt{k}}\right)\right)=i\,\sqrt{k}\,({\theta_1-\theta_2})+\psi_2\left(\mathbf{v}_1,\,\mathbf{v}_2\right)\,. \]

Given $\theta$ sufficiently small, $x\in X_{\boldsymbol{\nu}}$ and $\mathbf{v}\in T_xX_{\boldsymbol{\nu}}$, we note that by Corollary 2.2 in \cite{pao-IJM} we get
\[\tilde{\mu}_{\exp_T \left(-\theta\,\tilde{\boldsymbol{\xi}}_1\right)}\left(x+\mathbf{v}\right)=x+\left(\theta+\theta\cdot\omega\left(\frac{\boldsymbol{\xi}_{1,\,M}}{\lVert\Phi(m)\rVert},\,\mathbf{v}\right),\,\mathbf{v}-\theta\cdot\frac{\boldsymbol{\xi}_{1,\,M}}{\lVert\Phi(m)\rVert}\right)+O\left(\lVert(\theta,\,\mathbf{v})\rVert^2\right)   \]
and eventually, we note that 
\begin{align*} 
&F_k^{\mu}\left(\left(\frac{\theta_1}{\sqrt{k}},\,\frac{\mathbf{v}_1}{\sqrt{k}}\right),\,\left(\frac{\theta_2}{\sqrt{k}},\,\frac{\mathbf{v}_2}{\sqrt{k}}\right)\right) \\ \notag
& \qquad = F_k^{SZ}\left(\tilde{\mu}_{\exp_T \left(-\frac{\theta_1}{\sqrt{k}}\,{\tilde{\boldsymbol{\xi}}_1}\right)}\left(x-\frac{\mathbf{v}_1}{\sqrt{k}}\right),\,\tilde{\mu}_{\exp_T \left(-\frac{\theta_2}{\sqrt{k}}\,{\tilde{\boldsymbol{\xi}}_1}\right)}\left(x-\frac{\mathbf{v}_2}{\sqrt{k}}\right)\right) +O\left(1/\sqrt{k}\right) 
% \\ & \qquad = F_k^{SZ}\left(\mathfrak{n}\left(\frac{\theta_1}{k},\,\frac{\mathbf{v}_1}{\sqrt{k}}\right),\,\mathfrak{n}\left(\frac{\theta_2}{k},\,\frac{\mathbf{v}_2}{\sqrt{k}}\right)\right) +O\left(1/\sqrt{k}\right)\notag
\end{align*}
where
\[ \tilde{\boldsymbol{\xi}}_1=\frac{\boldsymbol{\nu}}{\lVert \boldsymbol{\nu} \rVert\cdot\lVert \Phi(m)\rVert}\,. \]

\section{Proof of Theorem~\ref{cor:aimi}}
\label{sec:proof}

Let us denote the dimension of $H_{k\boldsymbol{\nu}}(X)$ by $d_{k\boldsymbol{\nu}}$ and consider an orthonormal basis of $H_{k\boldsymbol{\nu}}(X)$ given by  ${s}_1^{k\boldsymbol{\nu}},\,\dots,\,{s}^{k\boldsymbol{\nu}}_{d_{k\boldsymbol{\nu}}}$ with respect to the standard $L^2$ product. 

We distinguish two cases: the stabilizer of $x$ is trivial or its cardinality is finite and equals $\lvert T_x\rvert>1$. First, let us consider a neighborhood of a point $x$ with $\lvert T_x\rvert = 1$.

Now, by definition, we can write in coordinates with respect to the dual basis 
\[\widetilde{\varphi_{k\boldsymbol{\nu}}}\,:\, X_{{\boldsymbol{\nu}}} \rightarrow H_{k\boldsymbol{\nu}}(X)^{\vee}\,,\quad x \mapsto  \left({s}_1^{k\boldsymbol{\nu}}(x),\,\dots,\,{s}^{k\boldsymbol{\nu}}_{d_{k\boldsymbol{\nu}}}(x)\right)\,, \]
%is explicitly given by
%\[\widetilde{\varphi_{k\boldsymbol{\nu}}}(x)=\sum_{j=1}^{d_{k\boldsymbol{\nu}}} s_{j}^{k\boldsymbol{\nu}}(x)(s_{j}^{k\boldsymbol{\nu}}) \in H_{k\boldsymbol{\nu}}(X)^*\cong\mathbb{C}^{d_{k\boldsymbol{\nu}}}\,\]
and we note that
\begin{equation} \label{phipi}
	\Pi_{k\boldsymbol{\nu}}(x,\,y):=
	\widetilde{\varphi_{k\boldsymbol{\nu}}}(x)\cdot \overline{\widetilde{\varphi_{k\boldsymbol{\nu}}}(y)}\,.
\end{equation}

We shall follow \cite{sz} and we consider the 2-form $\Omega$ on a sufficiently small neighborhood of the diagonal $(\mathbb{C}^{d_{k\boldsymbol{\nu}}}\setminus \{0\})\times (\mathbb{C}^{d_{k\boldsymbol{\nu}}}\setminus \{0\})$ given by
\[ \Omega=\frac{i}{2}\,\mathrm{d}^{(1)}\mathrm{d}^{(2)}\log \zeta\cdot \overline{\eta}  \]
where $\zeta$ and $\eta$ are the variables respectively on the first and second component of $$(\mathbb{C}^{d_{k\boldsymbol{\nu}}}\setminus \{0\})\times (\mathbb{C}^{d_{k\boldsymbol{\nu}}}\setminus \{0\})$$ and the formal differentials $\mathrm{d}^{(1)}$ and $\mathrm{d}^{(2)}$ denote the exterior differentiation on the first and second factors, respectively. On the diagonal of $(\mathbb{C}^{d_k}\setminus \{0\})\times (\mathbb{C}^{d_k}\setminus \{0\})$, we have
\[ j^*(\Omega)=\tilde{\omega}_{d_k}:=\frac{i}{2}\,\partial\overline{\partial}\log \lvert\zeta \rvert^2 \,\]
where $j$ is the diagonal embedding.
%where $\tilde{\omega}_{d_k}$ induces the Fubini-Study form on $\mathbb{CP}^{d_k-1}$. 

Furthermore define
\[\Phi_{k\boldsymbol{\nu}}=\widetilde{\varphi_{k\boldsymbol{\nu}}}\times \widetilde{\varphi_{k\boldsymbol{\nu}}}\,:\,X_{\boldsymbol{\nu}}\times X_{\boldsymbol{\nu}} \rightarrow \mathbb{C}^{d_{k\boldsymbol{\nu}}}\times \mathbb{C}^{d_{k\boldsymbol{\nu}}},\quad \Phi_{k\boldsymbol{\nu}}(x,y):=(\widetilde{\varphi_{k\boldsymbol{\nu}}}(x),\,\widetilde{\varphi_{k\boldsymbol{\nu}}}(y))\,.  \]

The proof of the following Lemma is left to the reader.

\begin{lem} \label{lem:com}
	$\Phi_{k\boldsymbol{\nu}}^{*}$ commutes with $\mathrm{d}^{(1)}$ and $\mathrm{d}^{(2)}$.
\end{lem}

Thus we can pullback $\Omega$ via $\Phi_{k\boldsymbol{\nu}}$, by \eqref{phipi} and Lemma \ref{lem:com} we can easily obtain 
\[\frac{1}{k}\Phi_{k\boldsymbol{\nu}}^*\Omega= \frac{i}{2k}\,\Phi_{k\boldsymbol{\nu}}^*\mathrm{d}^{(1)}\mathrm{d}^{(2)}\log \zeta\cdot \overline{\eta} = \frac{i}{2k}\,\mathrm{d}^{(1)}\mathrm{d}^{(2)}\Phi_{k\boldsymbol{\nu}}^*\log \zeta\cdot \overline{\eta} =
\frac{i}{2k}\,\mathrm{d}^{(1)}\mathrm{d}^{(2)}\log \Pi_{k\boldsymbol{\nu}}\,.  \] 

Consider a small open neighborhood $U\subset Y_{\boldsymbol{\nu}}=X_{\boldsymbol{\nu}}/\exp_T(i\,\boldsymbol{\nu}^0)$ of a point $y$ in $Y_{\boldsymbol{\nu}}$ and $Q^{-1}_{\boldsymbol{\nu}}(U)\subset X_{\boldsymbol{\nu}}$. Let $x\in X_{\boldsymbol{\nu}}$ such that $Q_{\boldsymbol{\nu}}(x)=y$. Since $Q_{\boldsymbol{\nu}}\,:\,X_{\boldsymbol{\nu}}\rightarrow Y_{\boldsymbol{\nu}}$ is a principal $\exp_T(i\boldsymbol{\nu}^0)$-bundle over $Y_{\boldsymbol{\nu}}$, we can assume $U$ to be sufficiently small so that there exists a trivializing section $s\,:\, U \rightarrow Q^{-1}_{\boldsymbol{\nu}}(U)$ with $V:=s(U)\subseteq Q^{-1}_{\boldsymbol{\nu}}(U)$.
%We would like to adapt a system of Heisenberg local coordinates centered at $x$ in $X_{\boldsymbol{\nu}}\subseteq X$ to $j_{\boldsymbol{\nu}}(V)$, where $j_{\boldsymbol{\nu}}\,:\, X_{\boldsymbol{\nu}}\rightarrow X$ is the embedding. A Heisenberg local chart determines a linear isometry $T_xX\cong \mathbb{R}\oplus \mathbb{C}^d$. 
Let us denote by $T^{\mathrm{t}}_mM^{\sharp}$ the space of vectors in $T_xX$ which are the horizontal lift of vectors in $T^{\mathrm{t}}_mM$. By \eqref{eq:decomposition}, where $x\in X_{\boldsymbol{\nu}}$ with $\pi(x)=m$ we have 
\[T_xX=T^{\mathrm{t}}_mM^{\sharp}\oplus T_xX_{\boldsymbol{\nu}}
\]
since, by formula \eqref{eq:inf}, for every $\boldsymbol{\xi}\in i\boldsymbol{\nu}^0$, we have $\boldsymbol{\xi}_X(x)= \boldsymbol{\xi}_M^{\sharp}(m)$. Since the vectors in $T^{\mathrm{v}}_mM=(i\boldsymbol{\nu})^0_M(m)$ are tangent to the fibers of the principal bundle $Q_{\boldsymbol{\nu}}\,:\,X_{\boldsymbol{\nu}}\rightarrow Y_{\boldsymbol{\nu}}$ in $X_{\boldsymbol{\nu}}\subseteq X$, then $s$ can be chosen so that a basis for $T_xV$ consists of a basis for $T^{\mathrm{hor}}_mM$ and $\boldsymbol{\xi}_{1,\,X}(x)$. Note that  that $T^{\mathrm{hor}}_xM^{\sharp}$ projects isomorphically to $T_nN_{\boldsymbol{\nu}}$ where $x$ is such that $p(x)=n$ and $p\,:\,X_{\boldsymbol{\nu}}\rightarrow N_{\boldsymbol{\nu}}$ is the projection.

Thus, let us consider $y\in U \subseteq Y_{\boldsymbol{\nu}}$ with $Q_{\boldsymbol{\nu}}(x)=y$. We compose with the diagonal map $\mathrm{diag}$ and we make use of Theorem \ref{thm:tang}, using Heisenberg coordinate, for $\mathbf{v}_1,\, \mathbf{v}_2\in T^{\mathrm{hor}}_mM$  we get
\begin{align*}
	\frac{1}{k}\widetilde{\varphi_{k\boldsymbol{\nu}}}^*{\omega}_{d_k}\big\vert_{V} &= \frac{i}{2k}\,\mathrm{diag}^*\mathrm{d}^{(1)}\mathrm{d}^{(2)}\log \Pi_{k\boldsymbol{\nu}}\left(x+\left(\frac{\theta_1}{\sqrt{k}},\,\frac{\mathbf{v}_1}{\sqrt{k}}\right),\,x+\left(\frac{\theta_2}{\sqrt{k}},\,\frac{\mathbf{v}_2}{\sqrt{k}}\right)\right) \\	&=
\frac{i}{2k}\,\mathrm{diag}^*\mathrm{d}^{(1)}\mathrm{d}^{(2)}\left[\frac{1}{\varsigma(m_x)}\,F^{\mu}_k\left(\left(\frac{\theta_1}{\sqrt{k}},\,\frac{\mathbf{v}_1}{\sqrt{k}}\right),\,\left(\frac{\theta_2}{\sqrt{k}},\,\frac{\mathbf{v}_2}{\sqrt{k}}\right)\right)\right] +O\left(1/\sqrt{k}\right)  .
\end{align*}
Thus, by the description of local coordinates on $V\subset X_{\boldsymbol{\nu}}$ for the circle bundle $\pi_{\boldsymbol{\nu}}\,:\,Y_{\boldsymbol{\nu}}\rightarrow N_{\boldsymbol{\nu}}$ at the end of section \ref{sctn:szego fio}, we obtain
\begin{align*}
	% &\frac{i}{2k}\,\mathrm{diag}^*\mathrm{d}^{(1)}\mathrm{d}^{(2)}\left[\frac{1}{\varsigma(m_x)}\, F_k^{SZ}\left(\tilde{\mu}_{\exp_T \left(-\theta_1\,\frac{\boldsymbol{\xi}_1}{k}\right)}\left(x-\frac{\mathbf{v}_1}{\sqrt{k}}\right),\,\tilde{\mu}_{\exp_T \left(-\theta_2\,\frac{\boldsymbol{\xi}_1}{k}\right)}\left(x-\frac{\mathbf{v}_2}{\sqrt{k}}\right)\right)\right] \\
	 \frac{1}{k}\widetilde{\varphi_{k\boldsymbol{\nu}}}^*{\omega}_{d_k}\big\vert_{V}& =\frac{i}{2k}\,\mathrm{diag}^*\mathrm{d}^{(1)}\mathrm{d}^{(2)}\left[\frac{1}{\varsigma(m_x)}\, \psi_2\left({\mathbf{v}_1},\,\mathbf{v}_2\right)\right] +O\left(1/\sqrt{k}\right)\\
	 &=\frac{1}{\varsigma(m_x)}\,\pi^*\omega{\vert_V} +O\left(1/\sqrt{k}\right)\,.
\end{align*}
Since in Heisenberg local coordinates $\omega$ is the standard form on $\mathbb{C}^{d-r+1}$ (where $r$ is the dimension of the torus $T$), see \cite{sz}, and by the discussion in Section \ref{sec:kahlerN}, see Lemma \ref{lem:2.1}, we get (2) in Theorem \ref{cor:aimi}. Thus, we also get $(1)$ since, as a consequence of (2) we get that $\mathrm{d}\varphi_{k\boldsymbol{\nu}}$ is injective for $k$ large.

Now, suppose that the stabilizer $\lvert T_x \rvert >1$, hence $n$ is a singular point in $N_{\boldsymbol{\nu}}$ and fix $x$ such that $p(x)=n$ where $p\,:\,X_{\boldsymbol{\nu}}\rightarrow N_{\boldsymbol{\nu}}$ is the projection. Since $N_{\boldsymbol{\nu}}$ is an orbifold there exists a local chart $(T_x,\, \widetilde{U_{n}})$.  By the Slice Theorem, see Appendix B in \cite{ggk}, a neighborhood of the orbit of $x$ is equivariantly diffeomorphic to a neighborhood of the zero section of the associated principal bundle $T\times_{T_x} D$ where $D$ is a sufficiently small neighborhood of the origin in $T_{m}^{\mathrm{hor}}M_{{\boldsymbol{\nu}}}$ with $\pi(x)=m$.

Now, $\Pi_{k\boldsymbol{\nu}}(x,\,x)$ can be zero for some $k$ even when $k$ is large. To solve this problem, we note that there exists a subsequence $k_n$ such that for each $x\in X_{\boldsymbol{\nu}}$
\[\sum_{t\in T_x}\chi_{k_n\boldsymbol{\nu}}(t)\neq 0\,. \]
In fact, we can choose $k_n$ to be divisible by the least common multiple of $\lvert T_x\rvert$, for $x\in X_{\boldsymbol{\nu}}$, (there are only finitely many possible stabilizers). As a consequence of 
\[ \widetilde{\varphi_{k\boldsymbol{\nu}}}(\tilde{\mu}_{t^{-1}}(x))= \chi_{k\boldsymbol{\nu}}(t)\sum_{j=1}^{d_{k\boldsymbol{\nu}}} s_{j}^{k\boldsymbol{\nu}}(x)(s_{j}^{k\boldsymbol{\nu}})=\hat{\mu}_t\left(\widetilde{\varphi_{k\boldsymbol{\nu}}}(x)\right)\,, \]
we note $\varphi_{k\boldsymbol{\nu}}$ is a $T_x$-equivariant map from an open neighborhood $\widetilde{U_{n}}\cong D$ of $\widetilde{n}_x$ to $\mathbb{CP}^{d_{k\boldsymbol{\nu}}-1}_{\chi_{k\boldsymbol{\nu}}}$, where in the orbifold local chart $\pi_{U_n}\,:\,(\widetilde{U_{n}},\,G_{{U_{n}}})\rightarrow U_n$ we have that $\pi_{U_n}(\widetilde{n}_x)=n_x$. 

Eventually, we shall prove that the maps $\varphi_{k\boldsymbol{\nu}}$ are embeddings for $k$ sufficiently large. In order to prove it, we study long-range and short-range injectivity as in \cite{sz}. Thus, given $x,\,y\in X_{\boldsymbol{\nu}}$ in the range where $\sqrt{k}\,\mathrm{dist}_X(T\cdot x,\,y)$ tends to infinity as $k\rightarrow +\infty$, it follows in a similar way as in \cite{sz} that $\varphi_{k\boldsymbol{\nu}}(n_x)\neq \varphi_{k\boldsymbol{\nu}}(n_y)$ as an easy consequence of \cite{pao-IJM}.

For short-range injectivity we consider $\mathbf{v}_k\in T_m^{\mathrm{hor}}M^{\sharp}$, where $m=\pi(x)$ and $x\in X_{\boldsymbol{\nu}}$, such that $0\neq\lVert \mathbf{v}_k \rVert\leq C$ and assume
\begin{align}
	{{\varphi_{k\boldsymbol{\nu}}}\left(n_{x+{\mathbf{v}_k}/{\sqrt{k}}}\right)} = {{\varphi_{k\boldsymbol{\nu}}}(n_x)}
	\label{eq:inj0}
\end{align}
where recall that $T^{\mathrm{hor}}_xM^{\sharp}$ projects isomorphically to $T_nN_{\boldsymbol{\nu}}$ where $x$ is such that $p(x)=n_x$ and $p\,:\,X_{\boldsymbol{\nu}}\rightarrow N_{\boldsymbol{\nu}}$ is the projection.
Thus, \eqref{eq:inj0} is equivalent to
\begin{equation}\frac{\widetilde{\varphi_{k\boldsymbol{\nu}}}(x)}{{\lVert \widetilde{\varphi_{k\boldsymbol{\nu}}}(x) \rVert}}\wedge \frac{\widetilde{\varphi_{k\boldsymbol{\nu}}}(x+\mathbf{v}_k/\sqrt{k})}{{\lVert \widetilde{\varphi_{k\boldsymbol{\nu}}}(x+\mathbf{v}_k/\sqrt{k}) \rVert}}=0 \,. \label{eq:inj}
\end{equation}
Thus, by \eqref{phipi}, equation \eqref{eq:inj} is equivalent to
\begin{equation}
	\left\lvert \frac{\Pi_{k\boldsymbol{\nu}}(x,\,x+\mathbf{v}_k/\sqrt{k})}{\sqrt{\Pi_{k\boldsymbol{\nu}}(x,\,x)}\cdot \sqrt{\Pi_{k\boldsymbol{\nu}}(x+\mathbf{v}_k/\sqrt{k},\,x+\mathbf{v}_k/\sqrt{k})} }\right\rvert = 1\,.  \label{eq:1}
\end{equation}

Define, for a sufficiently small $\epsilon$, $f\,:\,[-\epsilon,\,1+\epsilon]\rightarrow \mathbb{R}$ 
\[ f_k(\tau)=\frac{\lvert\Pi_{k\boldsymbol{\nu}}(x,\,x+\tau\,\mathbf{v}_k/\sqrt{k})\rvert^2}{{\Pi_{k\boldsymbol{\nu}}(x,\,x)}\cdot {\Pi_{k\boldsymbol{\nu}}(x+\tau\,\mathbf{v}_k/\sqrt{k},\,x+\tau\,\mathbf{v}_k/\sqrt{k})} }\]
and note that $f_k(0)=1$ and, by \eqref{eq:1}, $f_k(1)=1$ (the denominator is non-zero provided we pass to a subsequence). Furthermore, by \eqref{phipi} and the Cauchy–Schwartz inequality,  $0\leq f_k \leq 1$. Hence we have $f_k'(0)=0$ (note that $f$ is defined on $[-\epsilon,\,1+\epsilon]$).
%and $f_k'(1)=O(k^{-1})$. 
Thus, there exists $\tau_k\in (0,\,1)$ such that $f''_k(\tau_k)=0$.

Eventually, starting with \eqref{eq:1} and expanding in decreasing half-integer powers of $k$, 
\[ f_k(\tau)= \sum_{t\in T_x}\chi_{k\boldsymbol{\nu}}(t)\cdot e^{-\frac{\tau^2}{2\,\varsigma(m)}\,\lVert\mathbf{v}_k\rVert^2}\left[1+k^{-1/2}R(\tau,\,\mathbf{v}_k)\right]\]
where $R_3(\tau,\,\mathbf{v}_k)=O(1)$. Note that, since $f_k(1)=1$, we get $\lVert \mathbf{v}_k \rVert^2=O(k^{-1/2})$.
%Since $f_k'(1)=O(k^{-1})$, we get $\lVert \mathbf{v}_k \rVert^2=O(k^{-1/2})$. 
Thus, we obtain $$f''_k(\tau)=(-1+O(1))\sum_{t\in T_x}\chi_{k\boldsymbol{\nu}}(t)\cdot\lVert \mathbf{v}_k \rVert^2\,.$$ By evaluating at $\tau_k$, we get $\mathbf{v}_k=\boldsymbol{0}$.

\bigskip
\textbf{Acknowledgments:} I express my gratitude to the referee for valuable suggestions and corrections, which have contributed to the improvement of this paper. 

I would like to thank the referee of a previous version of this paper for providing several comments and for pointing out a mistake in a previous version of this paper that led to a change in the statement of Corollary 1.1. 

I would also like to express my gratitude to Prof. R. Paoletti for suggesting the study of immersion problems. This project was initiated during my postdoctoral fellowship at the National Center for Theoretical Sciences in Taiwan and was completed at the University of Milano-Bicocca. I would like to thank both institutions for their support.

Additionally, I extend my thanks to the Mathematics Institute of Universität zu K\"oln for their hospitality throughout my scholarship: the author is supported by INdAM (Istituto Nazionale di Alta Matematica) foreign scholarship.

\end{document}